\documentclass{iopart}
\usepackage{iopams}
\usepackage{geometry}                
\usepackage{graphicx}
\usepackage{epstopdf}
\usepackage{amsthm}
\usepackage[pagebackref,pdfusetitle]{hyperref}
\usepackage{url}
\usepackage[percent]{overpic}

\def\eqalign#1{\null\,\vcenter{\openup\jot \mathsurround=0pt \ialign{\strut
     \hfil$\displaystyle{##}$&$ \displaystyle{{}##}$\hfil \crcr#1\crcr}}\,}

\newtheorem{proposition}{Proposition}
\newtheorem{corollary}[proposition]{Corollary}
\def\R{\mathbb{R}}

\begin{document}

\title{Geometric properties of Kahan's method}

\author{Elena Celledoni$^1$, Robert I McLachlan$^2$, Brynjulf Owren$^1$ and G R W Quispel$^3$}

\address{$^1$ 	Department of Mathematical Sciences,
	NTNU,
	7491 Trondheim,
	Norway\eads{\mailto{elenac@math.ntnu.no}, \mailto{bryn@math.ntnu.no}}}
\address{$^2$ 	Institute of Fundamental Sciences,
	Massey University,
	Private Bag 11 222, Palmerston North 4442, New Zealand\ead{r.mclachlan@massey.ac.nz}}
\address{$^3$ 	Department of Mathematics,
	La Trobe University,
	Bundoora, VIC 3083, Australia\ead{r.quispel@latrobe.edu.au}}

\begin{abstract}
\noindent
We show that Kahan's discretization of quadratic vector fields is equivalent to a Runge--Kutta method. When the vector field is Hamiltonian on either a symplectic vector space or a Poisson vector space with constant Poisson structure, the map determined by this discretization has a conserved modified Hamiltonian and an invariant measure, a combination previously unknown amongst Runge--Kutta methods applied to nonlinear vector fields. This produces large classes of integrable rational mappings in two and three dimensions, explaining some of the integrable cases that were previously known.
\end{abstract}

\section{Introduction: Kahan's method for quadratic vector fields}
Consider a system of differential equations arising from a quadratic
vector field
\begin{equation}
\dot x = f(x) := Q(x)+Bx+c,\quad x\in\R^n,
\label{eq:quadvf}
\end{equation}
where $Q$ is an $\R^n$-valued quadratic form, $B\in\R^{n\times n}$, and $c\in\R^n$.
Consider the numerical integration method $x\mapsto x'$ with step size $h$ given by
\begin{equation}
\frac{x'-x}{h} = Q(x,x') + \frac{1}{2}B(x+x')+ c
\label{eq:kahanmap}
\end{equation}
where
\begin{equation}
Q(x,x') = \frac{1}{2}\left( Q(x+x') - Q(x) - Q(x')\right)
\end{equation}
is the symmetric bilinear form obtained from the quadratic form $Q$ by polarisation. We call (\ref{eq:kahanmap}) the {\em Kahan method}.
It is symmetric (i.e., self-adjoint), and, crucially, it is only linearly implicit, that is, 
$x'$ can be computed by solving a single linear system (because the right 
hand side of  (\ref{eq:kahanmap}) is linear in $x'$). The method (\ref{eq:kahanmap}) was introduced by Kahan in \cite{kahan} for two examples, a scalar Riccati equation and a 2-dimensional Lotka--Volterra system (\cite{kahan}, p. 14) and written down in the general form (\ref{eq:kahanmap}) in \cite{ka-li} (see also references therein).

Because of the different treatment of each term and the unusual treatment of the quadratic term, Kahan called (\ref{eq:kahanmap}) an `unconventional' method.

The map obtained from applying the Kahan method to various quadratic
vector fields $f$ has been shown to be completely integrable in a number of
cases (see \cite{hi-ki}, \cite{ho-pe}, \cite{pe-pf-su}, and references therein).  In most cases, the conserved 
quantities depend on the step size $h$. At
present there is no single `integrability mechanism' known which accounts for all integrable cases.

In this paper we show that the Kahan method is a Runge--Kutta method. As such it shares a number of features with all Runge--Kutta methods: it has  a B-series, it is affine covariant, and it preserves all affine symmetries and all linear integrals of $f$ automatically. As a symmetric linear method it preserves all affine reversing symmetries of $f$ automatically, and the B-series of its modified vector field contains only even powers of $h$.

We then consider the case that $f$ is a Hamiltonian vector field on either a symplectic vector space or a Poisson vector space with constant Poisson structure, in any dimension $n$.
 We show that in this case the Kahan map has a conserved quantity that converges to the Hamiltonian of the vector field as $h\to 0$. It also has a conserved measure which converges to the Euclidean measure as $h\to 0$. These general properties explain some of the integrable cases considered in \cite{pe-pf-su}.

\section{Kahan's method as a Runge--Kutta method}
\begin{proposition}
The Kahan method coincides with the Runge--Kutta method
\begin{equation}
\label{eq:kahanrk}
\frac{x'-x}{h} = -\frac{1}{2}f(x) + 2 f\Big(\frac{x+x'}{2}\Big) - \frac{1}{2}f(x')
\end{equation}
restricted to quadratic vector fields.
\end{proposition}
\begin{proof}
We have
\begin{equation}
\eqalign{\frac{x'-x}{h} &=
Q(x,x') + \frac{1}{2}B(x+x')+ c \cr
&= \frac{1}{2}\left(Q(x+x') - Q(x) - Q(x')\right) + 
\frac{1}{2}B(x+x') + c \cr
&= \frac{1}{2}\Big(4 Q\Big(\frac{x+x'}{2}\Big) - Q(x) - Q(x')\Big)
- \frac{1}{2}B x + 2 B\frac{x+x'}{2} - \frac{1}{2} B x' + c \cr
&= -\frac{1}{2}f(x) + 2 f\Big(\frac{x+x'}{2}\Big)
 - \frac{1}{2}f(x').}
\end{equation}
\end{proof}
(Many other Runge--Kutta methods also coincide with the Kahan method when restricted to quadratic vector fields. In this paper, we restrict our attention to (\ref{eq:kahanrk}).)
As already noted by Kahan \cite{ka-li}, the Kahan method also coincides
with a certain Rosenbrock method on quadratic vector fields, for expanding
in Taylor series about $x$ gives
\begin{equation}
\eqalign{\frac{x'-x}{h} 
&= -\frac{1}{2}f(x) + 2 f\Big(\frac{x+x'}{2}\Big) - \frac{1}{2}f(x')\cr
&= -\frac{1}{2}f(x) + 2\Big(f(x) + \frac{1}{2}f'(x)(x'-x) + \frac{1}{8}f''(x)(x'-x,x'-x)\Big)\cr
&\hspace{1.6cm}
-\frac{1}{2}\Big(f(x) + f'(x)(x'-x) + \frac{1}{2}f''(x)(x'-x,x'-x)\Big) \cr
&= f(x) + \frac{1}{2}f'(x)(x'-x) \cr
}
\end{equation}
so
\begin{equation}
\label{eq:rosenbrock}
\frac{x' - x}{h} = \Big(I - \frac{h}{2}f'(x)\Big)^{-1}f(x).
\end{equation}
From the symmetry of the method, or by expanding instead
around $x'$, Kahan's method can also be written
\begin{equation}
\label{eq:rosenbrock2}
\frac{x' - x}{h} = \Big(I + \frac{h}{2}f'(x')\Big)^{-1}f(x').
\end{equation}

The B-series of this method is
$$ x' = x + \sum_{k=0}^\infty \frac{h^{k+1}}{2^k}f'(x)^k f(x),$$
that is, it contains only tall trees. (For nonquadratic vector fields,
the methods (\ref{eq:rosenbrock}) and (\ref{eq:kahanrk}) are not
necessarily equivalent.)

The Runge--Kutta method (\ref{eq:kahanrk}) has 3 stages and Butcher
tableau
\begin{center}
\begin{tabular}{c|ccc}
0 & 0 & 0 & 0 \\[1mm]
$\frac{1}{2}$ & $-\frac{1}{4}$& 1 & $-\frac{1}{4}$\\[1mm]
1 & $-\frac{1}{2}$ & 2 & $-\frac{1}{2}$ \\[2mm]
\hline\\[-3mm]
& $-\frac{1}{2}$ & 2 & $-\frac{1}{2}$ \\
\end{tabular}
\end{center}
The modified vector field of the Kahan method applied to quadratic vector fields can be calculated using standard methods \cite{ha-lu-wa}. Its first few terms are 
$$\eqalign{
f +&  \frac{h^2}{12}\left(-2 f''(f,f) + f'f'f)\right) 
+ \frac{h^4}{240} \Big(3 f'^4 f - 2 f'^2f''(f,f)  \cr
& - 6 f'f''(f,f'f)-8 f''(f,f'^2f)+12 f''(f,f''(f,f)) + 4 f''(f'f,f'f)\Big)
+\dots.}$$
A calculation using conjugation by B-series, considering only quadratic vector fields, now yields the following result. We omit the details.

\begin{proposition}
\label{prop2}
Kahan's method applied to general quadratic fields has order 2 and
is conjugate to symplectic up to order 4. It is not conjugate by B-series to a method
of order greater than 2 or conjugate-symplectic by B-series to order higher than 4.
\end{proposition}

\section{Conservative properties of Kahan's method}

We now consider the conservative properties of the Kahan method in the
case of canonical Hamiltonian systems $\dot x = J^{-1} \nabla H(x)$
where $H:\R^n\to \R$ is the Hamiltonian or energy of the system.
First, note that the method (\ref{eq:kahanrk}) is the $a=-1/2$ member of the
class of Runge--Kutta methods
\begin{equation}
\label{eq:symrk}
\frac{x'-x}{h} = a f(x) + (1-2a) f\Big(\frac{x+x'}{2}\Big) + a f(x').
\end{equation}
These are all symmetric, A-stable, and second order. Some other members
of this family are also known to have conservative properties:
\begin{enumerate}
\item When $a=0$, we have the midpoint rule. It is symplectic for canonical Hamiltonian systems. Because it is symplectic, it conserves the Euclidean
measure. When the Hamiltonian $H$ is analytic,
the method has a formal invariant $\widetilde H = H+ \sum_{k=1}^\infty h^k H_k$. When
$H$ is quadratic (i.e. when $f$ is linear) this series converges
to give a conserved quantity of the method. 
\item When $a=1/2$, we have the trapezoidal rule. It is conjugate
to the midpoint rule (the conjugacy being half an Euler step), and 
so it is also conjugate to symplectic and hence conserves a measure
close to the Euclidean measure, and it also has a formal invariant close to $H$.
\item When $a=1/6$, we have `Simpson's method' \cite{cetal}, so-called because
the right hand side of (\ref{eq:symrk}) is Simpson's quadrature
of $A=\int_0^1 f(\xi x + (1-\xi)x')\, d\xi$, appearing in the average vector field method $\frac{x'-x}{h} = A$, which conserves the Hamiltonian
in canonical Hamiltonian systems. Simpson's method preserves
quartic Hamiltonians exactly because it coincides with the average vector field method in that case. It is not conjugate to symplectic in the
sense of B-series \cite{ha-lu-wa}.
\end{enumerate}

\begin{proposition}
Kahan's method has a conserved quantity given by the modified Hamiltonian
\begin{equation}
\label{eq:Ht}
\widetilde H(x) := H(x) + \frac{1}{3}h \nabla H(x)^T \Big( I - \frac{1}{2}h f'(x)\Big)^{-1}f(x) 
\end{equation}
for all cubic Hamiltonian systems on symplectic vector spaces
and on all Poisson vector spaces with constant Poisson structure.
The modified Hamiltonian is (i) a rational function of $x$; (ii) an even function of $h$; and (iii)
given by a convergent series of elementary Hamiltonians containing only even-order tall 
trees.
\end{proposition}

\begin{proof}
We first consider the homogeneous case, i.e., we let $f 
= K \nabla H(x)$ where $K$ is an arbitrary (not necessarily invertible) constant antisymmetric matrix
and $H(x) = C(x,x,x)$ where $C$ is a symmetric trilinear form.
Note that $\nabla H(x)^T v = 3 C(x,x,v)$ for all $x$, $v\in\R^n$.
For any of the methods (\ref{eq:symrk}), we have (writing
$\bar x = (x+x')/2$)
$$\eqalign{
0 &= h \left(a \nabla H(x) + (1-2a) \nabla H(\bar x) + a\nabla H(x')\right)^T
K^T \left(a \nabla H(x) + (1-2a) \nabla H(\bar x) + a\nabla H(x')\right) \cr
&=h \left(a f(x) + (1-2a) f(\bar x) + a f(x')\right)^T
\left(a \nabla H(x) + (1-2a) \nabla H(\bar x) + a\nabla H(x')\right) \cr
&= (x'-x)^T(a \nabla H(x) + (1-2a) \nabla H(\bar x) + a \nabla H(x') )\cr
&= 3a C(x,x,x'-x) + \frac{3}{4}(1-2a) C(x+x',x+x',x'-x) + 3a C(x',x',x'-x) \cr
&= \frac{3}{4}\left( (2a+1)(C(x',x',x')-C(x,x,x)) + (6 a -1 )(C(x,x,x')-C(x,x',x'))\right). \cr
}$$

The case $a=\frac{1}{6}$ is Simpson's method, confirming that $H(x)$ is conserved in that case.

For Kahan's method, $a=-\frac{1}{2}$,  and we have from Eqs. 
(\ref{eq:rosenbrock},\ref{eq:rosenbrock2}) that
$$ \eqalign{
x'' - x &= h \Big((I + \frac{1}{2}h f'(x'))^{-1} + (I-\frac{1}{2}h f'(x'))^{-1}\Big)f(x') \cr
&= 2 h \Big(I - \frac{1}{4}h^2 f'(x')^2\Big)^{-1}f(x'). \cr
}$$
Therefore
\begin{equation}
\label{eq:Cs}
\eqalign{
C(x',x',x'') - C(x,x,x') &= C(x',x',x'') - C(x',x',x) \cr
&= C(x',x',x''-x) \cr
&= \frac{1}{3}\nabla H(x')^T (x''-x) \cr
&= \frac{1}{3}\nabla H(x') ^T 2 h\Big(I- \frac{1}{4}h^2 f'(x')^2\Big)^{-1}f(x') \cr
& = \frac{2}{3} h \nabla H(x')^T \Big(\sum_{n=0}^\infty (h f'(x')/2)^{2n}\Big)f(x')\cr
&= \frac{2}{3}h \sum_{n=0}^\infty \nabla H(x')^T\left[(K H''(x'))^{2n} K\right]\nabla H(x') \cr
& = 0. 
}
\end{equation}
because each matrix in square brackets is antisymmetric. (Each term is an elementary Hamiltonian corresponding 
to a superfluous tall tree.) The expression is a rational function of $x$ and $h$ so if its Taylor series in $h$ is zero, the function is zero. Therefore the  Kahan method has a
first integral $C(x,x,x')$. This can also be written in the  
symmetric form
$$C(x,x,x') = (C(x,x,x') + C(x,x',x'))/2 = C(x,(x+x')/2,x')$$
or explicitly as a function of $x$ as
$$ \eqalign{
C(x,x,x') &= C(x,x,x + h \Big(I-\frac{1}{2}h f'(x)\Big)^{-1}f(x))\cr
&= H(x) + \frac{1}{3}h \nabla H(x)^T \Big(I - \frac{1}{2}h f'(x)\Big)^{-1}f(x) \cr
&= \widetilde H(x).}$$
As in  (\ref{eq:Cs}), $\nabla H(x)^T f'(x)^{2n}f(x) = 0$ for all $n$, so we can also write
$\widetilde H(x)$ in a form manifestly even in $h$,
$$\widetilde H(x) = H(x) + \frac{1}{6}h^2 \nabla H(x)^T \Big(I-\frac{1}{2}h^2 f'(x)^2\Big)^{-1}f'(x) f(x).$$

Now consider the case that $H(x)$ is cubic but not homogeneous. We extend it to a homogeneous function $\overline H(x_0,x_1,\dots,x_n)$ so 
that $\overline H(1,x_1,\dots,x_n) = H(x_1,\dots,x_n)$, and
extend $K$ to $\overline K$ by adding a zero initial row and column, so that $\dot x_0 = 0$.
The linear integral $x_0$ is conserved by Kahan's method,
and $\overline K \nabla\overline H(1,x) = K\nabla H(x)$, so the
modified Hamiltonian of Kahan's method for $\overline K \nabla\overline H(x)$ reduces to a modified Hamilton of Kahan's method for $K\nabla H(x)$ given
by the same formula (\ref{eq:Ht}) as in the homogeneous case.
\end{proof}

The Kahan map and its conserved quantity $\widetilde H$ are  rational functions of $x$ whose degrees are described in the following proposition. When $K$ has full rank, for $n=2$ (resp. $3,4$), $\widetilde H$ is degree 3 over degree 2 (resp. degree 5/2, degree 5/4) and the Kahan map is degree 2/2 (resp. degree 3/2, degree 4/4). In the planar case, Kahan's method gives a rational map with cubic invariant curves. We conjecture that the dynamics of the Kahan map in two dimensions  is related to the abelian group structure of elliptic curves (and in higher dimensions, to that of abelian varieties) as is the case for planar QRT maps \cite{duistermaat}.

\def\adj{\mathop{\rm adj}}

\begin{proposition}
Let $H$ be a cubic in $\R^n$ and let $K$ be a rank $k$ antisymmetric $n\times n$ matrix. 
\begin{enumerate}
\item
The degree of the denominator of $\widetilde H$ is at most $k$ and the degree of the numerator of $\widetilde H$ is at most $k+3$. When $k=n$ the degree of the numerator of $\widetilde H$ is at most $k+1$.
\item
The degree of the denominator of the Kahan map is at most $k$ and the degree of the numerator is at most $k+1$. When $k=n$ the degree of the numerator is at most $k$.
\end{enumerate}
\end{proposition}
\begin{proof}
\begin{enumerate}
\item
Because the method is linearly covariant we can assume without loss of generality that $K$ is in its Darboux normal form
$$K = \left(\matrix{L & 0 \cr  0 & 0 }\right),\quad L = \left(\matrix{0 & I \cr -I & 0}\right)$$
where $L$ is $k\times k$. Numbering the blocks of $K$ 1 and 2, 
the denominator of $\widetilde H$ is equal to 
$$\eqalign{
\det(I-\frac{1}{2}h f') &= \det(I-\frac{1}{2}h K H''(x)) \cr
&= \det\left(I - \frac{1}{2}h\left( \matrix{L & 0 \cr  0 & 0}\right)
\left(\matrix{H_{11}(x) & H_{12}(x) \cr H_{21}(x) & H_{22}(x)}\right)\right)\cr
&= \det\left( \matrix{I - \frac{1}{2}h L H_{11}(x) & -\frac{1}{2}h L H_{12}(x) \cr 0 & I }\right) \cr
&= \det(I-\frac{1}{2}h L H_{11}(x)).\cr
}$$
Each entry of the matrix $L H_{11}(x)$ is linear in $x$ so the final determinant has degree at most $k$.
Next, we write 
$$ \widetilde H(x) = \frac{H(x) \det(I-\frac{1}{2}h f'(x)) + \frac{1}{3}h (\nabla H(x))^T \adj(I-\frac{1}{2}h f'(x))f(x)}{\det(I-\frac{1}{2}h f'(x))}$$
where $\adj(A) = A^{-1} \det A$ is the adjoint of $A$.
The first term in the numerator has degree at most $k+3$, $\nabla H$ and $f$ have degree at most 2, and
$\adj(I-\frac{1}{2}h f'(x))f(x)=\adj(I-\frac{1}{2}h L H_{11}(x))L H_1(x)$ where each entry in $\adj(I-\frac{1}{2}h f'(x))$ is the determinant of a $(k-1)\times(k-1)$ matrix whose entries are linear in $x$. Hence the degree of the numerator of $\widetilde H(x)$ is at most $k+3$.

Finally we consider the case $k=n$. Since $k$ is even, $n$ must be even. First consider the case that $H$ is a homogeneous cubic. Then $\nabla H(x) = \frac{1}{2}H''(x) x$ and $H(x) = \frac{1}{6} x^T H''(x) x$. Thus in this case we have
$$\eqalign{
6 \widetilde H(x) &= 6 H(x) + 2 h \nabla H(x)^T \Big( I - \frac{1}{2}h f'(x)\Big)^{-1}f(x) \cr
&= x^T \left( H''(x) + H'' (x)(2 h^{-1} I - K H''(x))^{-1} K H''(x)\right) x \cr
&= x^T \Big( H''(x)(I-\frac{1}{2}h K H''(x))^{-1}\Big) x. \cr
}$$
Expanding the matrix inverse using Cramer's rule now shows that the degree of the numerator is at most $k+2$. The terms of degree $k+2$ come from terms in the minors of $I-\frac{1}{2}h K H''(x)$ of degree $k-1$ in $x$. Every $x_i$ in the matrix $I-\frac{1}{2}h K H''(x)$ is multiplied by $h$, thus these terms also have degree $k-1$ in $h$. However, $\widetilde H(x)$ is an even function of $h$ and so these terms must sum to zero. Thus the degree of the numerator is at most $k+1$. 

When $H(x)$ is a nonhomogeneous cubic, the terms of degree $k+3$ in the numerator of $\widetilde H(x)$ come from the cubic terms in $H(x)$ only, and thus vanish as in the homogeneous case. The terms of degree $k+2$ are odd in $h$ and hence vanish as before.
\item
The proof for the general case follows as above. For the case $k=n$, we first consider the case that $H$ is a homogeneous cubic. Then $f(x) = \frac{1}{2}f'(x)x$ and the Kahan map can be written
$$x' = x + \Big(I-\frac{h}{2}f'(x)\Big)^{-1}f(x) = \Big(I-\frac{h}{2}f'(x)\Big)^{-1}x.$$
Expanding the matrix inverse using Cramer's rule now shows that the degree of the numerator is at most $k+1$. In the nonhomogeneous case, the terms of degree $k+2$ in the numerator come from the cubic terms in $H$ only, and thus vanish as in the homogeneous case.
\end{enumerate}
\end{proof}

Examples suggest that there are no other values of $n$ or $k$ other than $n=k$ which lead to a reduction in degree. 

\begin{proposition}
Kahan's method preserves the measure 
$$ \frac{dx_1\wedge\dots\wedge dx_n}{\det(I - \frac{1}{2} h f'(x))}$$
for all cubic Hamiltonians on symplectic vector spaces and on Poisson vector spaces with constant Poisson structure.
\end{proposition}

\def\half{{\textstyle\frac{1}{2}}}
\def\quarter{{\textstyle\frac{1}{4}}}

\begin{proof}
Let $A=\frac{\partial x'}{\partial x}$ be the Jacobian of the Kahan mapping.
Differentiating the mapping (\ref{eq:symrk}) gives
$$\eqalign{
 \frac{A-I}{h} &= a f'(x) + \frac{1}{2}(1-2a) f'((x+x')/2)(I+A) + a f'(x') A\cr
 &= a f'(x) + \frac{1}{4} (1-2a) (f'(x) + f'(x'))(I+A) + a f'(x')A\cr
 }$$
because $f'(x)$ is linear in $x$. Solving for $A$,
$$ A  = \left(I - (\quarter + \half a)hf'(x') - (\quarter - \half a)hf'(x) \right)^{-1}\left(I + (\quarter+ \half a)hf'(x) + (\quarter - \half a)hf'(x') \right).
$$
Now $f'(x) = K H''(x) =: K S$ where $S$ is symmetric. 
From Sylvester's determinant theorem, $\det(I+KS) = \det(I+SK) = \det((I+SK)^T) = \det(I - K S)$.
The sum of such matrices has the same property, so
$$ \det A = \frac{\det(I-(\quarter+\half a)hf'(x)-(\quarter - \half a)hf'(x'))}
{\det(I-(\quarter+\half a)hf'(x')-(\quarter - \half a)hf'(x))}.
$$
This yields invariant measures $m(x)dx_1\wedge\dots\wedge dx_n$ in 3 cases:
\begin{enumerate}
\item when $a=-1/2$ (Kahan's method), $m(x) = 1/\det(I - \frac{1}{2}h f'(x))$;
\item when $a=0$ (midpoint rule), $m(x) = 1$;
\item when $a=1/2$ (trapezoidal rule), $m(x) = \det(I - \frac{1}{2}h f'(x))$.
\end{enumerate}
\end{proof}

By {\em integrable symplectic map} we adopt the definition of Bruschi et al. \cite{bruschi}: a symplectic map on a $2n$-dimensional symplectic manifold is integrable if it has $n$ functionally independent integrals in involution. We will say that leaf-preserving Poisson maps are integrable if the the map is integrable on each leaf.

\begin{corollary}
\label{cor:integrable}
Kahan's method yields an integrable mapping of the plane
when applied to any canonical Hamiltonian system in the plane
with cubic Hamiltonian.
Kahan's method yields an integrable mapping of $\R^3$
when applied to any Poisson system on $\R^3$ with constant
Poisson structure and any cubic Hamiltonian.
\end{corollary}

\begin{proof}
A measure and a first integral are sufficient for integrability in the plane. 
The odd-dimensional case with constant $K$ has a linear Casimir which is conserved by the method, reducing the situation in this case to two dimensions on each level set of the 
Casimir.
\end{proof}

\begin{corollary}
When $n=2$ and $H$ is a homogeneous cubic, $\widetilde H(x) = H(x)/\det(I-\frac{1}{2}h f'(x))$ and Kahan's method
preserves the $h$-independent measure $(dx_1\wedge dx_2)/H(x)$.
\end{corollary}
\begin{proof}
We have
$$\eqalign{
6 \det(I-\frac{1}{2}h f'(x)) \widetilde H(x)  &= \det(I-\frac{1}{2}h f'(x))  x^T \Big( H''(x)(I-\frac{1}{2}h K H''(x))^{-1}\Big) x \cr
&= x^T \Big( H''(x) \adj(I-\frac{1}{2}h K H''(x))\Big)x \cr
&=x^T \Big(H''(x) (I - \adj(\frac{1}{2}h K H''(x))\Big)x \quad \hbox{\rm (because $n=2$)}\cr
&=x^T \Big(H''(x) - \frac{1}{2}hH''(x) (K H''(x))^{-1}\det(K H''(x))\Big)x\cr
&=x^T \Big(H''(x) - \frac{1}{2}h K \det(K H''(x))\Big)x\cr
&=x^T H''(x) x \quad  \hbox{\rm (because $K^T=-K$)}\cr
&= 6 H(x).\cr
}$$
Any map that preserves a measure $\mu(x)$ and an integral $I(x)$ also preserves the measure $I(x)\mu(x)$. Taking $\mu = dx_1\wedge dx_2/\det(I-\frac{1}{2} h f'(x))$ and $I(x) = 1/\widetilde H(x)$ gives the $h$-independent measure $(dx_1\wedge dx_2)/H(x)$.
\end{proof}

\begin{figure}[pt]
\begin{center}
\begin{overpic}[width=6cm]{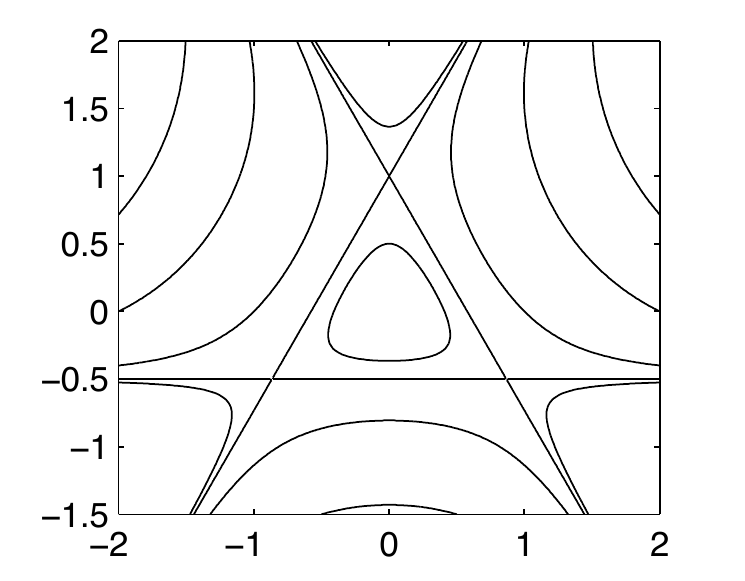}
\put(75,2){$q$}\put(0,65){$p$}\end{overpic}
\begin{overpic}[width=6cm]{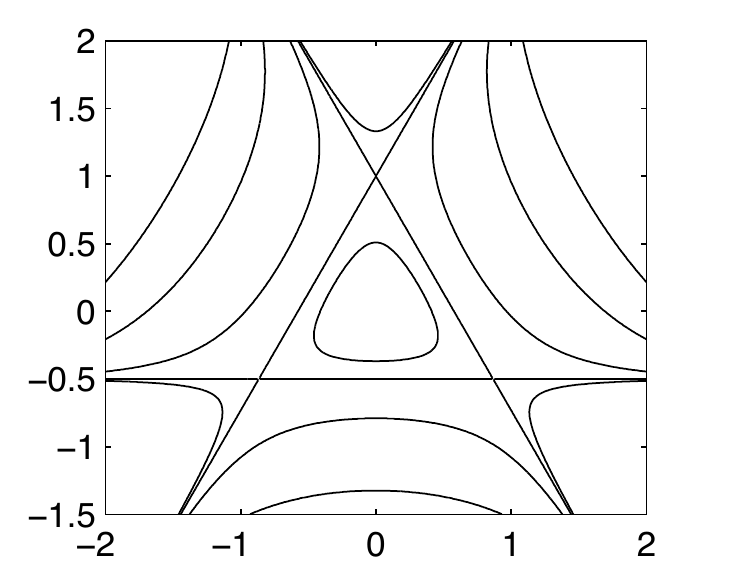}
\put(75,2){$q$}\put(0,65){$p$}\end{overpic}
\begin{overpic}[width=6cm]{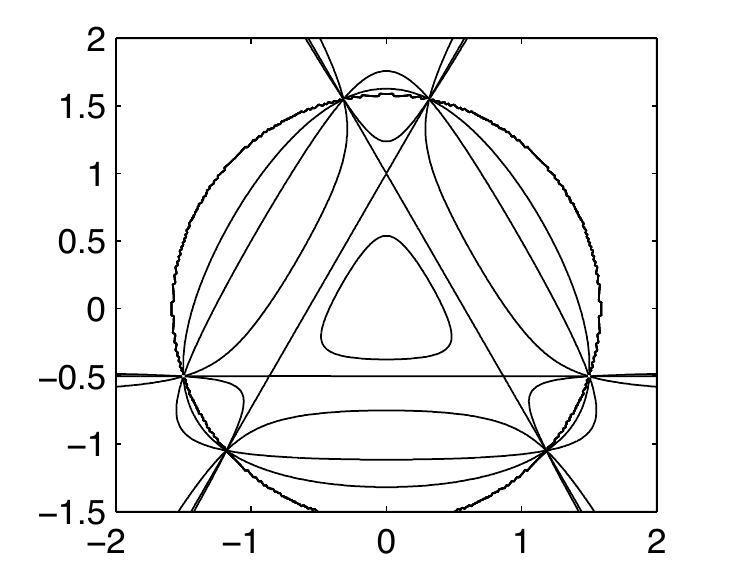}
\put(75,2){$q$}\put(0,65){$p$}\end{overpic}
\begin{overpic}[width=6cm]{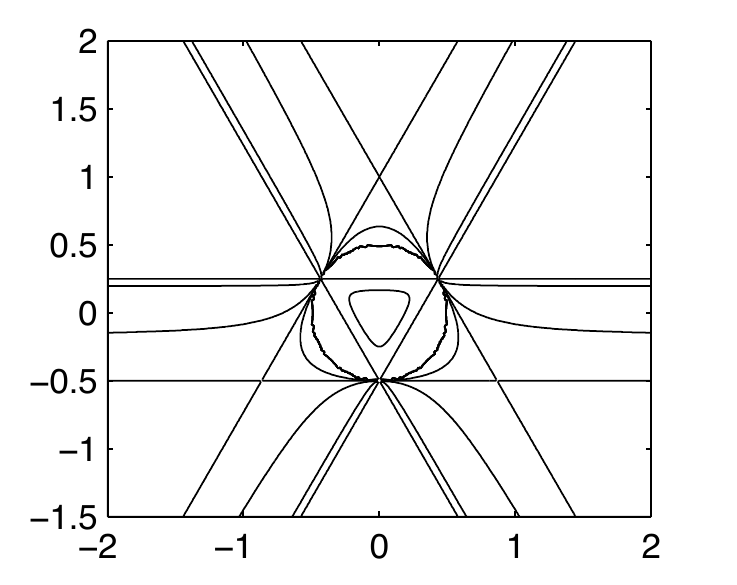}
\put(75,2){$q$}\put(0,65){$p$}\end{overpic}
\caption{\label{fig:henon}
Top left: Level sets of $H = \frac{1}{2}(q^2+p^2) + q^2 p - \frac{1}{3}p^3$ (the so-called H\'enon--Heiles potential). Same level sets of the conserved quantity $\widetilde H$ 
of Kahan's method for $h=1/3$ (top right); $h=2/3$ (bottom left) (the jagged circle $q^2+p^2=\frac{1}{4}+\frac{1}{h^2}=1.58$ indicates $\tilde H = \infty$, on which initial conditions are mapped to infinity---for $h=1/3$ the circle has radius 3.04 and is out of view); and $h\to\infty$ (bottom right). Note that Kahan's method preserves the 3-fold discrete symmetry of $H$, because as a Runge--Kutta method it preserves all affine symmetries.
}
\end{center}
\end{figure}

\begin{figure}[pt]
\begin{center}
\begin{overpic}[width=6cm]{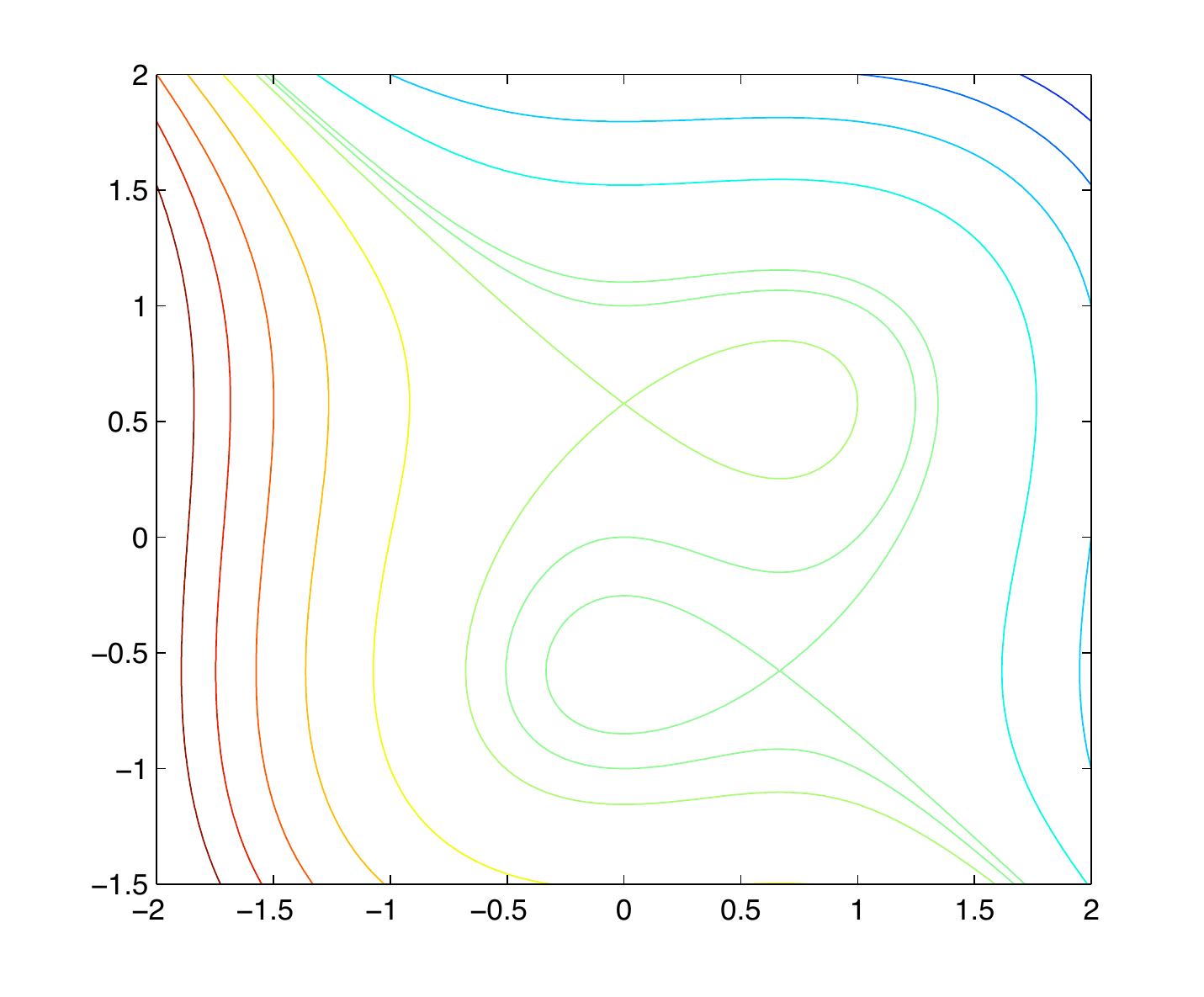}
\put(75,2){$q$}\put(0,65){$p$}\end{overpic}
\begin{overpic}[width=6cm]{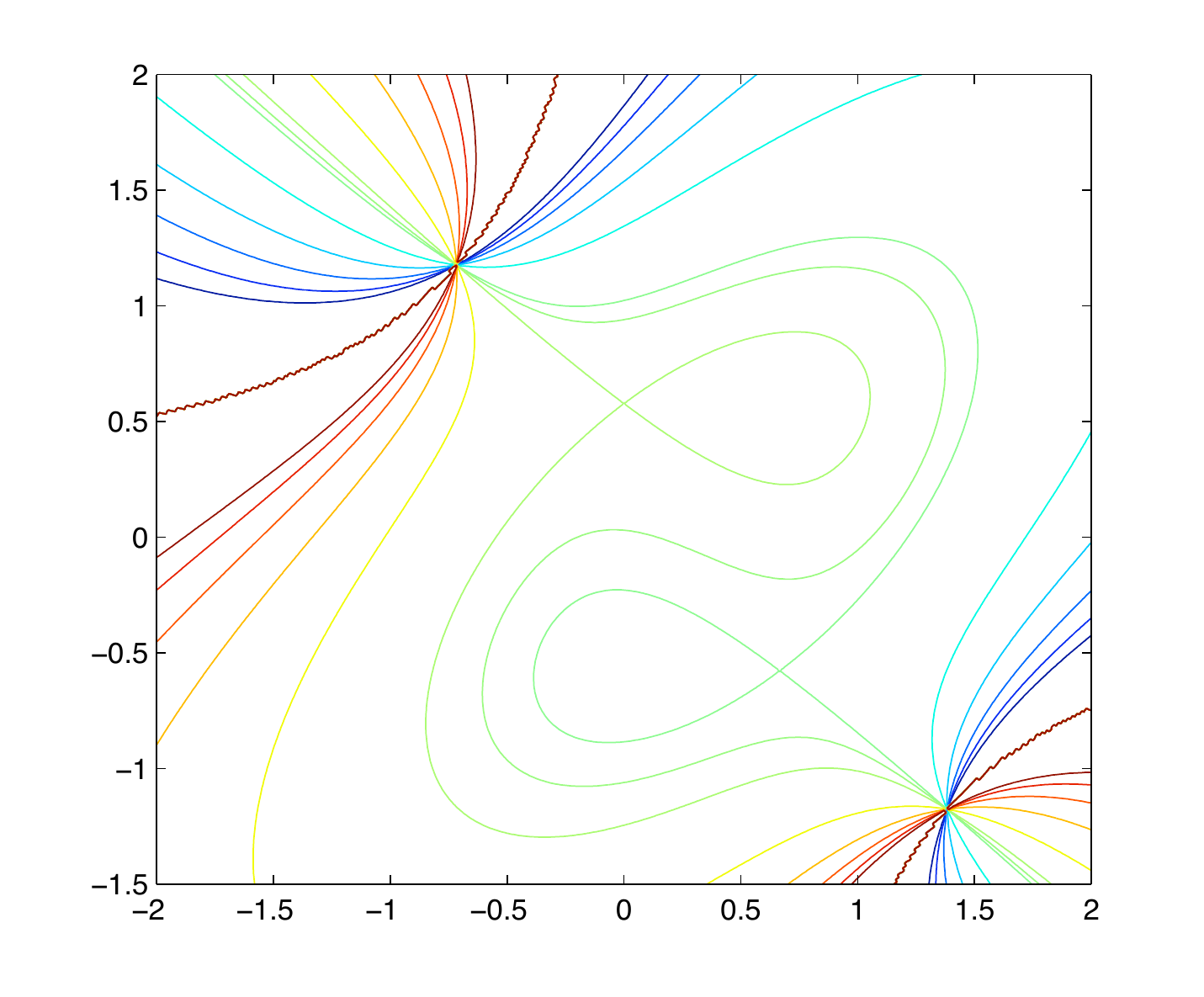}
\put(75,2){$q$}\put(0,65){$p$}\end{overpic}
\caption{\label{fig:nons}
Left: Level sets of $H=p-p^3 + q^2 - q^3$. 
Right: 
Level sets of the conserved quantity $\widetilde H$ 
of Kahan's method for $h=0.3$. Later numerical experiments use initial conditions
inside the separatrix attached to $(q,p)=(0,1/\sqrt{3})$.
}
\end{center}
\end{figure}

\begin{figure}[pt]
\begin{center}
\includegraphics{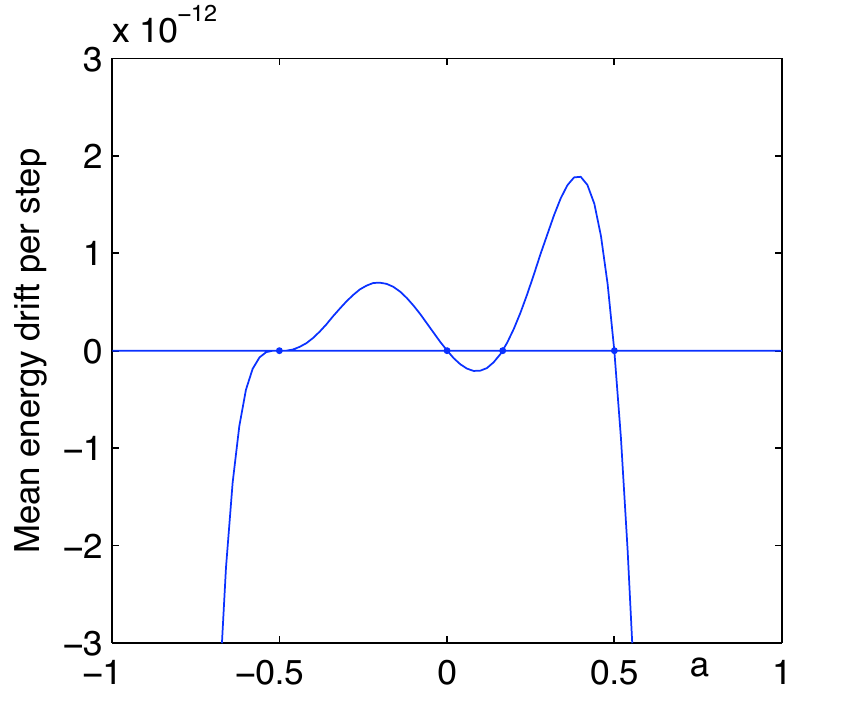}
\end{center}
\caption{\label{fig:family}Measured rate of energy drift for $H=p-p^3+q^2-q^3$ for Runge--Kutta methods 
$x'=x+h(a f(x) + (1-2a) f((x+x')/2) + a f(x'))$, varying the parameter $a$. The step size
is $h=0.3$ and the initial condition is $q=0.323$, $p=1/\sqrt{3}$. All methods have
an approximate modified energy up to  $h^4$. The energy drift is measured by
fitting a straight line to this modified energy over $2\times 10^6$ time steps. Only the four
methods identified by the analysis ($a=-1/2$, Kahan; $a=0$, midpoint; $a=1/6$, Simpson; and
$a=1/2$, trapezoidal) show no energy drift by this measure.
}
\end{figure}

\begin{figure}[pt]
\begin{center}
\includegraphics[width=9cm]{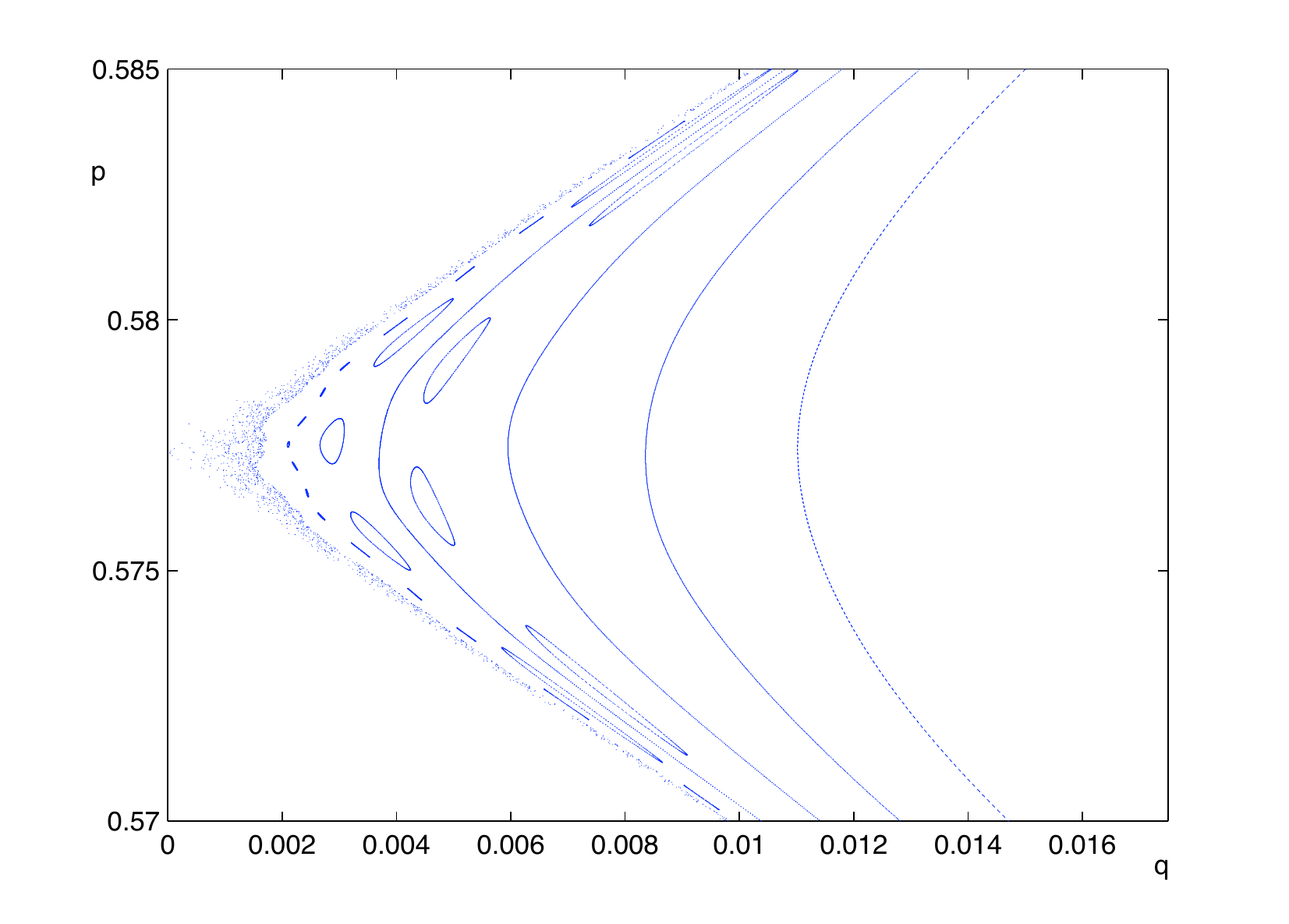}
\end{center}
\caption{\label{fig:midpoint}Portion of the phase portrait of the midpoint rule with step size $h=0.3$ applied to
$H = p-p^3+q^2-q^3$. The observed chaotic bands
and island chains indicate that it does not have a conserved quantity.
}
\end{figure}

\begin{figure}[pt]
\begin{center}
\includegraphics[width=9cm]{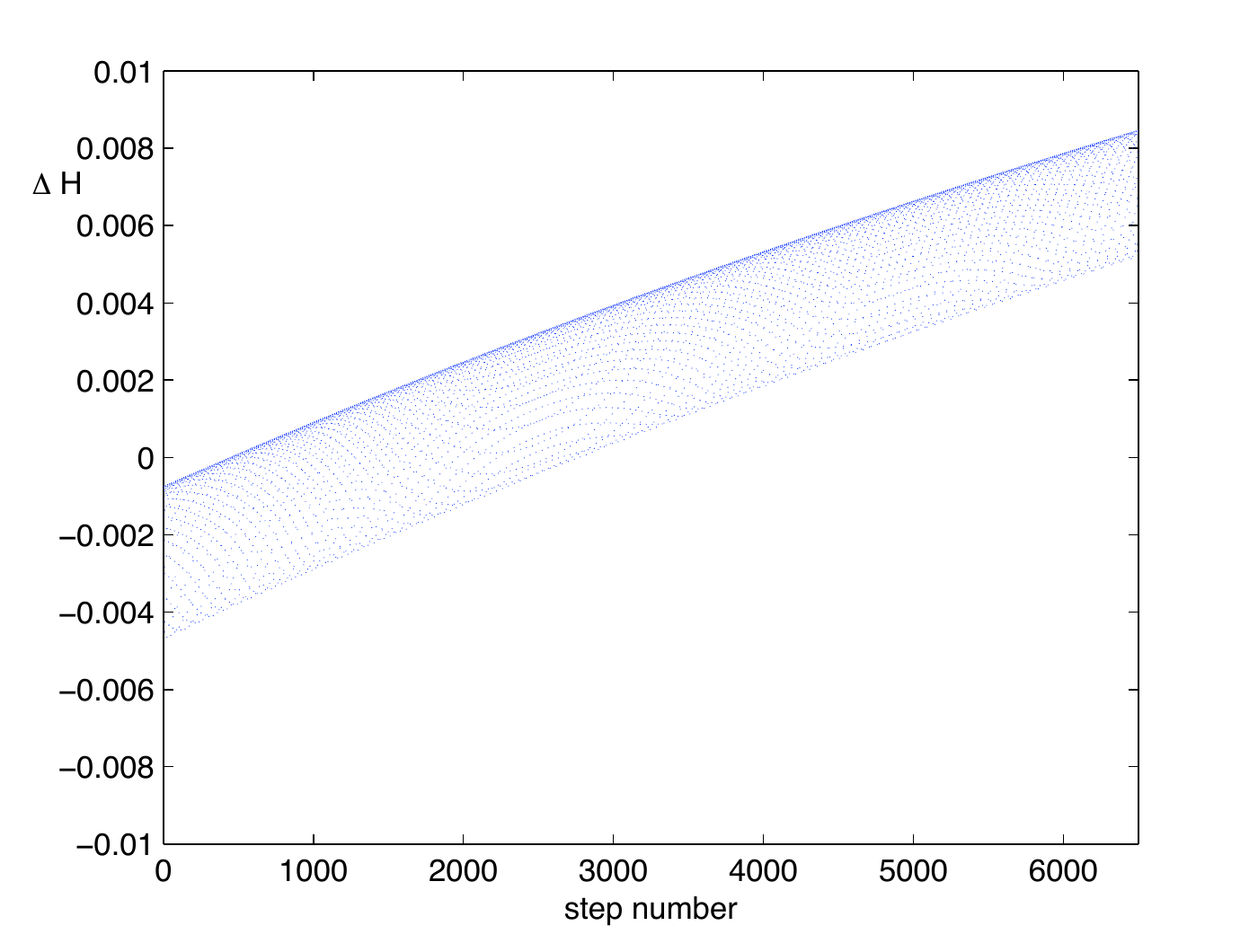}
\end{center}
\caption{\label{fig:suzuki}Suzuki's 3-stage, 4th-order composition applied
to Kahan's method shows a comparatively rapid energy
drift, indicating that there is no conserved quantity.
Here $H=p-p^3+q^2-q^3$, $h=0.2$, and $(q_0,p_0)=(0.323,1/\sqrt{3})$.
}
\end{figure}

\section{Discussion}

Level sets of $\widetilde H$ are shown in Figure \ref{fig:henon} for $H = \frac{1}{2}(q^2+p^2) + q^2 p - \frac{1}{3}p^3$. Notice that the separatrices persist (and are unchanged) for all $h$, but that the singular set $q^2 + p^2 = \frac{1}{4} + \frac{1}{h^2}$ moves in from infinity as $h$ increases and alters the topology of the level sets. For $h<\sqrt{4/3}$ the topology of the
bounded orbits is unaltered. 

The bounded orbits of Figure \ref{fig:henon} are symmetric, so for the following numerical experiments we used $H=p-p^3 + q^2 - q^3$, which has bounded, nonsymmetric orbits,  an elliptic fixed point at $(q,p)=(2/3,1/\sqrt{3})$, and 
a separatrix meeting $(q,p)=(0,1/\sqrt{3})$. Level sets of $\widetilde H$ for this case are shown in Figure \ref{fig:nons}.

Numerical experiments strongly indicate that the following observations hold. 
\begin{enumerate}
\item No other method of the family (\ref{eq:symrk}) has
a modified Hamiltonian when $H$ is cubic, apart from the known
cases $a=0$, $\pm \frac{1}{2}$, and $\frac{1}{6}$ (see Figure \ref{fig:family}).
\item The midpoint and trapezoidal rules do not have a
first integral for all cubic $H$ (even though they do 
have a formal invariant close to $H$) (see Figure \ref{fig:midpoint}).
\item Simpson's method is not measure-preserving
for all cubic $H$. (For $H=p-p^3+q^2-q^3$, a numerical
calculation finds eigenvalues $1, \lambda$ of periodic
orbits, with $\lambda\ne 1$, contradicting measure preservation.)
\item Kahan's method does not preserve any symplectic form in dimension $\ge 4$ for all cubic $H$. (A numerical calculation of periodic points finds eigenvalues that do not occur in $\lambda$, $1/\lambda$ pairs. Proposition \ref{prop2} establishes this for a limited class of symplectic forms.)
\item Compositions of Kahan's method with different
step sizes do not have a modified Hamiltonian when $H$ is cubic
(see Figure \ref{fig:suzuki}).
\end{enumerate}

Our results are significant and novel for the study of both the integrability of the mappings produced by Kahan's method and for the study of the geometric properties of Runge--Kutta methods:
\begin{itemize}
\item First, our results explain the integrability of the map obtained when Kahan's method is applied to {\em some} of the examples
of \cite{pe-pf-su}: their Eq. (4.2) ($H=y^2/2 - 2 x^3 + \alpha x$); Eq. (5.4) ($H = y(3x^2-y)$); Eq. (8.1) (Volterra chain in
$\R^3$, $H = x_1 x_2 x_3$, constant $K$, integral $H_2$ in Eq. (8.6) is a function of our $\widetilde H$ and the Casimir
$x_1+x_2+x_3$); Eq. (9.1) (Dressing chain in $\R^3$, $H = (x_1+x_2)(x_2+x_3)(x_3+x_1)-\sum_i \alpha_i x_i$, constant $K$). Our results explain the invariant measure and cubic integral of their Eq. (11.1) (three wave system in ${\mathbb C}^3$, $H=z_1z_2z_3+\bar z_1\bar z_2\bar z_3$), the invariant measure for the family of systems in their Prop. 1, and the linear integrals throughout \cite{pe-pf-su}. 
\item Second, our results (e.g. Corollary \ref{cor:integrable}) systematically produce new integrable cases of Kahan's method.
\item Third, we have shown that Kahan's method in dimension 4 and greater provides examples
of maps with nonlinear integrals and conserved measures {\em unrelated} (in general) to integrability or obvious symmetries, again 
a novel feature. For example, 
our results imply that Kahan's application of the method to the Korteweg--de Vries equation in \cite{ka-li} preserves a measure and a modified energy (but the higher order compositions of the method in \cite{ka-li} probably do not).
\item Fourth, we have shown that Kahan's method has novel properties previously unknown amongst Runge--Kutta methods, indeed amongst all B-series. It is known that B-series methods
cannot conserve the measure $dx_1\wedge\dots\wedge dx_n$
even for linear vector fields \cite{is-qu-ts}; Kahan's method
circumvents this by conserving a modified measure. 
It is a novel conjugate-to-energy preserving method for cubic $H$.
In the plane, it is also conjugate to symplectic. Thus, while
no conjugate to symplectic methods are known that are also
energy preserving in general, here we have one that preserves
at least a modified energy, and preserves it exactly (not
merely as a formal invariant). 
\end{itemize}

On the other hand, there are open questions in all of these areas. While it was already suggested in \cite{pe-pf-su} that there could be an underlying `integrability mechanism' unifying the integrable cases, here we have unified only some of these. In addition the hoped-for unification should now be extended to include non-integrable cases preserving a measure and/or some integrals as well. On the numerical side, it is not known precisely which Runge--Kutta or B-series methods share the properties of Kahan's method, if any are higher order integrators, or if any are conservative for nonquadratic (e.g. other polynomial) vector fields.

\section*{Acknowledgements}
This research was supported by a  \href{http://wiki.math.ntnu.no/crisp}{Marie Curie International Research Staff Exchange Scheme Fellowship} within the \href{http://cordis.europa.eu/fp7/home_en.html}{7th European Community Framework Programme}, and by the Marsden Fund of the Royal Society of New Zealand and the Australian Research Council. We would like to thank the referees for their careful reading of the manuscript.

\end{document}